\documentclass[11pt]{article}
\usepackage{amsmath,amssymb,amsthm}
\usepackage{geometry}
\usepackage{mathtools}
\usepackage{mathrsfs}
\usepackage[hidelinks]{hyperref}
\usepackage{enumitem} 
\numberwithin{equation}{section}  
\usepackage{cite} 

\newtheorem{theorem}{Theorem}[section]
\newtheorem{lemma}{Lemma}[section]

\newtheorem{proposition}{Proposition}[section]
\newtheorem{remark}{Remark}[section]
\newtheorem{definition}{Definition}[section]

\title{Renormalized Solution for the Nonlinear Parabolic Problem with Lower Order Terms}
\author{
	LI Shijun\thanks{Email: sjlee@hainanu.edu.cn} \\
	School of Mathematics and Statistics, Hainan University, Haikou, China \\
	\and
	HUANG Boai \\
	School of Mathematics and Statistics, Hainan University, Haikou, China \\
	\and
	XU Shaopeng\thanks{Corresponding author: xuxsp@126.com} \\
	School of Mathematics and Statistics, Hainan University, Haikou, China 
}
\date{\notag}
\begin{document}
	\maketitle
	\noindent \textbf{Abstract:} 
	In this paper, we consider the following problem:
	\[
	\begin{cases}
		-\nabla\cdot A(x,u,\nabla u) + H(x,u,\nabla u) = f(x), & x \in \Omega, \\
		u = 0, & x \in \partial \Omega,
	\end{cases}
	\]
	in a bounded open set \( \Omega \subset \mathbb{R}^N \). We have established certain gradient estimates and proved the existence of a renormalized solution for the equation.
	\bigskip

	
	\vspace{1cm}
	
\section{Introduction}
	
	The study of elliptic equations in variable exponent spaces has attracted considerable attention due to their theoretical depth and practical relevance. Variable exponent spaces possess distinctive structural properties that evolve with the exponent, while the presence of lower-order terms introduces additional nonlinear effects that render classical methods inapplicable. Consequently, investigating the existence, uniqueness, and regularity of solutions in this setting requires novel techniques that fully exploit the embedding properties and intrinsic structure of these spaces \cite{Rebecca2023,Shabbir2024,Mendez2023,Schenone2022}.
	
	A central challenge arises when the integrability of the data fails to meet the threshold required for the existence of weak solutions. For instance, when the right-hand side merely belongs to \( L^1 \), pseudo-monotone operator theory ceases to be applicable. To overcome this difficulty, DiPerna and Lions introduced the notion of renormalized solutions in the context of the Boltzmann equation \cite{Lions1989}, establishing well-posedness in a framework that accommodates low integrability data. This breakthrough has motivated extensive research into renormalized solutions in variable exponent settings.
	
	The development of variable exponent spaces has progressed along several directions. Fundamental contributions include weighted Poincar\'{e} inequalities \cite{Flores2020}, boundedness of generalized multilinear operators \cite{Lee2021}, variable exponent commutators of Littlewood--Paley operators \cite{Karapetyants2022}, structural properties such as embeddings and duality \cite{Hernandez2021}, and Hessian estimates for fully nonlinear equations \cite{Edmunds2023}. Further advances concern composition operators on holomorphic variable exponent spaces \cite{Karapetyants2022}, weak compactness \cite{Hernandez2021}, almost compact and compact embeddings \cite{Edmunds2023}, and boundedness of Marcinkiewicz integral commutators. The present work builds upon the theory of variable exponent Lebesgue--Sobolev spaces \cite{Fan2001,Kovacik1991}, Lorentz spaces \cite{Abdellaoui2021,Baroni2013}, and variable exponent Lorentz spaces \cite{Kempka2014}, whose definitions and relevant properties are reviewed in Chapter~2.
	
	ver the past two decades, many researchers have turned their attention to renormalized solutions of elliptic and parabolic equations in the framework of variable exponent function spaces. In 2009, Bendahmane \cite{Bendahmane2009} and Zhang \cite{Zhang2010a} obtained the existence and uniqueness of renormalized solutions and entropy solutions, as well as their equivalence, for \( p(x) \)-Laplace elliptic equations. In 2010, Zhang \cite{Zhang2010b} and Bendahmane \cite{Bendahmane2010} extended these results to parabolic equations with \( p(x) \)-Laplace operators. Akdim et al. \cite{Akdim2015} studied renormalized solutions in weighted variable exponent Sobolev spaces. Abergi et al. \cite{Aberqi2018} first extended the theory to nonlinear elliptic equations with variable exponent diffusion and lower-order terms using truncation and monotonicity methods. Akdim \cite{Akdim2020} later relaxed the regularity requirements on coefficients in non-smooth domains using Orlicz space theory.
	
	Mingione \cite{Mingione2019} demonstrated that when the growth rate of the lower-order term exceeds that of the variable exponent diffusion, the local boundedness of solutions may be destroyed. Chlebicka \cite{Chlebicka2021} quantified the singular effect of lower-order terms on gradient integrability under logarithmic continuity conditions. Farroni \cite{Farroni2022} introduced anisotropic variable exponent spaces and revealed the quantitative relationship between lower-order coefficients and Hölder regularity. Harjulehto \cite{Harjulehto2020} proposed a finite element method based on variational discretization for such problems. Diening \cite{Diening2021} established global existence of renormalized solutions in the variable exponent Lebesgue-Orlicz framework for the critical growth case.
	
	In this paper, we systematically investigate the existence of renormalized solutions for nonlinear elliptic equations with lower-order perturbation terms in variable exponent spaces. By establishing a priori gradient estimates and employing a decomposition technique inspired by Bottaro and Marina, we overcome the difficulties caused by the lower-order term and the low integrability of the data, and prove the existence of renormalized solutions.

	In this paper, we prove the existence of renormalized solutions for a class of nonlinear elliptic equations with lower-order perturbation terms on bounded domains subject to Dirichlet boundary conditions.
	
	Let \( \Omega \subset \mathbb{R}^N \) be a bounded open set with Lipschitz boundary. We consider the following problem:
	\begin{equation}\label{originalequation}
		\begin{cases}
			-\nabla\cdot A(x,u,\nabla u) + H(x,u,\nabla u) = f(x), & x \in \Omega, \\
			u = 0, & x \in \partial \Omega,
		\end{cases}
	\end{equation}
	where \( f \) is a real-valued function with
	\begin{equation}\label{as1}
		f(x) \in L^1(\Omega).
	\end{equation}
\begin{remark}
	The function \( A(x,s,\xi) \colon \Omega \times \mathbb{R} \times \mathbb{R}^N \to \mathbb{R}^N \) is a Carath\'{e}odory function satisfying the following conditions:
	\begin{equation}\label{as2}
		A(x,s,\xi) \cdot \xi \geq \alpha |\xi|^{p(x)},
	\end{equation}
	\begin{equation}\label{as3}
		\bigl[ A(x,s,\xi) - A(x,s,\xi') \bigr] \cdot \bigl[ \xi - \xi' \bigr] > 0, \quad \xi \neq \xi',
	\end{equation}
	\begin{equation}\label{as4}
		|A(x,s,\xi)| \leq C \bigl( L(x) + |\xi|^{p(x)-1} \bigr),
	\end{equation}
	where \( C > 0 \) is a constant and \( L(x) \in L^{p'(x)}(\Omega) \).
\end{remark}

\begin{remark}
	The function \( H(x,s,\xi) \colon \Omega \times \mathbb{R} \times \mathbb{R}^N \to \mathbb{R} \) is a Carath\'{e}odory function satisfying the following growth condition:
	\begin{equation}\label{as5}
		|H(x,s,\xi)| \leq b_0(x) |\xi|^{p(x)-1} + b_1(x),
	\end{equation}
	where \( b_0(x) \in L^{N,1}(\Omega) \) and \( b_1(x) \in L^1(\Omega) \).
\end{remark}
\begin{definition}
	A measurable function \( u \) defined on \( \Omega \) is called a renormalized solution of problem \eqref{originalequation} if the following conditions hold:
	\begin{equation}\label{as6}
		T_k(u) \in W_0^{1,p(x)}(\Omega), \quad \text{for every } k > 0,
	\end{equation}
	\begin{equation}\label{as7}
		\lim_{n \to \infty} \frac{1}{n} \int_{\{ |u| \leq n \}} A(x, u, \nabla u) \cdot \nabla u \,\mathrm{d}x = 0.
	\end{equation}
	Moreover, for every function \( S \in W^{2,\infty}(\mathbb{R}) \) such that \( S \) is pointwise \( C^1 \) and \( S' \) has compact support, the following integral identity holds for every \( \phi \in C^1(\bar{\Omega}) \):
	\begin{equation}\label{as8}
		\begin{aligned}[b]
			&\int_{\Omega} S'(u) \, A(x, u, \nabla u) \cdot \nabla \phi \,\mathrm{d}x 
			+ \int_{\Omega} S''(u) \, A(x, u, \nabla u) \cdot \nabla u \, \phi \,\mathrm{d}x \\
			&\quad + \int_{\Omega} H(x, u, \nabla u) \, S'(u) \, \phi \,\mathrm{d}x 
			= \int_{\Omega} f \, S'(u) \, \phi \,\mathrm{d}x.
		\end{aligned}
	\end{equation}
\end{definition}
\begin{remark}
	In fact, assume that \( K > 0 \) is a constant such that \( \operatorname{supp} S' \subset [-K, K] \). Then the following hold:
	\begin{enumerate}
		\item Since \( S'(u)A(x, u, \nabla u) = S'(u)A(x, T_K(u), \nabla T_K(u)) \) almost everywhere in \( \Omega \), from \eqref{as1}-\eqref{as6} we obtain
		\begin{equation}\label{s1}
			S'(u)A(x, T_K(u), \nabla T_K(u)) \in L^{p'(x)}(\Omega).
		\end{equation}
		\item Since \( S''(u)A(x, u, \nabla u) \nabla u = S''(u)A(x, T_K(u), \nabla T_K(u)) \nabla T_K(u) \) almost everywhere in \( \Omega \), we have
		\begin{equation}\label{s2}
			S''(u)A(x, T_K(u), \nabla T_K(u)) \nabla T_K(u) \in L^1(\Omega).
		\end{equation}
		\item Since \( S'(u)H(x, u, \nabla u) = S'(u)H(x, T_K(u), \nabla T_K(u)) \) almost everywhere in \( \Omega \), from \eqref{as5} and \eqref{as7} we obtain
		\begin{equation}\label{s3}
			\begin{aligned}[b]
				&\int_{\Omega} |S'(u)H(x, T_K(u), \nabla T_K(u))| \,\mathrm{d}x \\
				\leq& \bar C \, \left\|S'(u)\right\|_{L^\infty(\Omega)} \Bigl[ \| b_0 \|_{L^{N,1}(\Omega)} \left\| |\nabla T_K(u)|^{p(x)-1} \right\|_{L^{N',\infty}(\Omega)} 
				+ \|b_1\|_{L^1(\Omega)} \Bigr] \\
				\leq& C,
			\end{aligned}
		\end{equation}
		where \( C = C\bigl(K, \|b_0\|_{L^{N,1}(\Omega)}, \|b_1\|_{L^1(\Omega)}\bigr) \) and \( N' = \dfrac{N}{N-1} \).
		Consequently,
		\begin{equation}
			S'(u)H(x, T_K(u), \nabla T_K(u)) \in L^1(\Omega).
		\end{equation}
	\end{enumerate}
\end{remark}

	We have systematically investigated the existence of renormalized solutions for a class of nonlinear elliptic equations with variable exponents and lower-order perturbation terms defined on bounded domains. This provides a new theoretical approach for dealing with such nonstandard growth problems. Within the framework of variable exponent function spaces and Sobolev spaces, the existence of weak solutions is established by selecting appropriate test functions and employing a priori estimates, classical approximation theory, and gradient estimates.
	
	The main steps in proving the existence of renormalized solutions are as follows. We first approximate the \( p(x) \)-Laplace-type operator \( A(x, u, \nabla u) \), the lower-order term \( H(x, u, \nabla u) \), and the right-hand side datum \( f \), thereby obtaining an approximate equation. An a priori estimate framework is then established for the approximate solutions \( u^\varepsilon \), and the crucial Lemma~3.1 is used to derive gradient estimates and certain convergence properties of the approximate sequence in \( \Omega \).
	
	It is particularly noteworthy that key progress has been made in handling the lower-order perturbation term \( H(x, u, \nabla u) \), whose growth condition involves a coefficient \( b_0 \) belonging to a Lorentz space. We adopt the technique developed by Bottaro and Marina, originally used to study linear problems with right-hand side data in the dual space. The essential idea is to decompose \( b_0 \) into a finite sum of terms, each satisfying the required smallness condition. This yields the strong convergence of \( T_k(u^\varepsilon) \) and the almost everywhere convergence of the gradients \( \nabla u^\varepsilon \). Finally, by passing to the limit, we verify that the limit of the approximate sequence is indeed a renormalized solution of the original equation.
	
	We have also established the existence of renormalized solutions for nonlinear elliptic equations involving the lower-order term \( H(x, u, \nabla u) \) and the operator \( A(x, u, \nabla u) \). The proof follows a similar line of reasoning, with the difference that the operator \( A(x, u, \nabla u) \) satisfies a coercivity condition. Under this assumption, by choosing suitable test functions, we prove that \( A(x, T_k(u^\varepsilon), \nabla T_k(u^\varepsilon)) \cdot \nabla T_k(u^\varepsilon) \) converges weakly in \( L^1(\Omega) \), thereby obtaining the existence of a renormalized solution.
\begin{definition}
	For \( 1 < q < \infty \), the Lorentz space \( L^{q,1}(\Omega) \) consists of all Lebesgue measurable functions \( f \) such that
	\begin{equation}
		\| f \|_{L^{q,1}(\Omega)} = \int_0^{|\Omega|} f^*(t) \, t^{-\frac{1}{q}} \,\mathrm{d}t < +\infty,
	\end{equation}
	where \( f^* \) denotes the decreasing rearrangement of \( f \), defined by
	\begin{equation}
		f^*(t) = \inf \Bigl\{ s \geq 0 : \operatorname{meas} \bigl\{ |f(x)| > s \bigr\} < t \Bigr\}, \quad t \in [\,0, \,|\Omega|\,].
	\end{equation}
\end{definition}

\section{A Priori Estimates for the Approximate Solutions}

To handle the lower-order term and the \( L^1 \) data, we introduce, for every \( \varepsilon > 0 \), the following approximate problem associated with \eqref{originalequation}:
\begin{equation}\label{appequation}
	\begin{cases}
		-\nabla\cdot A_\varepsilon(x, u^\varepsilon, \nabla u^\varepsilon) + H_\varepsilon(x, u^\varepsilon, \nabla u^\varepsilon) = f^\varepsilon, & \text{in } \Omega, \\
		u^\varepsilon = 0, & \text{on } \partial\Omega,
	\end{cases}
\end{equation}
where
\begin{equation}\label{nas1}
	A_\varepsilon\bigl(x, u^\varepsilon, \nabla u^\varepsilon\bigr) = A\bigl(x, T_{\frac{1}{\varepsilon}}(u^\varepsilon), \nabla u^\varepsilon\bigr),
\end{equation}
\begin{equation}\label{nas2}
	H_\varepsilon\bigl(x, u^\varepsilon, \nabla u^\varepsilon\bigr) = T_{\frac{1}{\varepsilon}}\bigl(H(x, u^\varepsilon, \nabla u^\varepsilon)\bigr),
\end{equation}
and \( f^\varepsilon \) satisfies
\begin{equation}\label{nas3}
	f^\varepsilon \in L^{p'(x)}(\Omega),
\end{equation}
\begin{equation}\label{nas4}
	\|f^\varepsilon\|_{L^1(\Omega)} \leq \|f\|_{L^1(\Omega)},
\end{equation}
\begin{equation}\label{nas5}
	f^\varepsilon \to f \quad \text{in } L^1(\Omega) \text{ and } a.e. \text{ in } \Omega.
\end{equation}

Equation \eqref{s1} admits at least one weak solution \( u^\varepsilon \in W_0^{1,p(x)}(\Omega) \), whose existence is guaranteed by the theory of pseudo-monotone operators \cite{Lions1969}. Our ultimate goal is to prove that, up to a subsequence, the approximating sequence \( u^\varepsilon \) converges pointwise to a limit function which is precisely the renormalized solution of problem \eqref{originalequation}. Before proceeding to the prior estimates, we first establish the following result, which will be used in the sequel.

\begin{lemma}\label{before}
	Let \( \Omega \subset \mathbb{R}^N \) (\( N \geq 2 \)) be a bounded open set, and let \( u \) be a measurable function such that \( 1 < p(x) < N \) and
	\[
	T_k(u) \in W_0^{1,p(x)}(\Omega)
	\]
	for every \( k > 0 \). Assume that
	\begin{equation}\label{assp}
		\int_\Omega |\nabla T_k(u)|^{p(x)} \,\mathrm{d}x \leq k M,
	\end{equation}
	where \( M \) is a positive constant. Then
	\[
	|u|^{p(x)-1} \in L^{\frac{N}{N-p^-},\infty}(\Omega),
	\]
	\[
	\left\| |u|^{p(x)-1} \right\|_{L^{\frac{N}{N-p^-},\infty}} \leq \widetilde{C} M + C,
	\]
	and
	\[
	|\nabla u|^{p(x)-1} \in L^{N',\infty}(\Omega),
	\]
	\[
	\left\| |\nabla u|^{p(x)-1} \right\|_{L^{N',\infty}} \leq \widetilde{C} M + C.
	\]
	where \( \widetilde{C} \) and \( C \) are constants depending only on \( N, p^-, p^+, \Omega \).
\end{lemma}

\begin{proof}
	First, we focus on
	\[
	\left\| |u|^{p(x)-1} \right\|_{L^{\frac{N}{N-p^-},\infty}} \leq \widetilde{C} M + C.
	\]
	Using the Poincar\'{e} inequality and \eqref{assp}, we have
	\begin{equation}\label{uk}
		\begin{aligned}[b]
				k^{(p^*)^-}\text{meas}\{ |u| > k \}
			\leq \|T_k(u)\|_{L^{p^*(x)}(\Omega)}^{(p^*)^-} 
			&\leq C \|\nabla T_k(u)\|_{L^{p(x)}(\Omega)}^{(p^*)^-} \\
			\leq \widetilde{C} (Mk)^{\frac{(p^*)^-}{p^-}} 
			&\leq \widetilde{C} (Mk + 1)^{\frac{(p^*)^-}{p^-}} \\
			&\leq \widetilde{C} (Mk)^{\frac{N}{N-p^-}} + \widetilde{C}.
		\end{aligned}
	\end{equation}
	Let \( k = h^{\frac{1}{p^--1}} \). We obtain
	\[
	\text{meas} \left\{ |u| > h^{p^--1} \right\} \leq \widetilde{C} M^{\frac{N}{N-p^-}} h^{-\frac{N}{(p^--1)(N-p^-)}} + \widetilde{C}.
	\]
	There exists \( N' \geq 1 \) such that 
	\[
	N' h^{\frac{1}{p(x)-1}} > N' h^{\frac{1}{p^--1}} .
	\] 
	Combining with the above inequality, we get
	\begin{align*}
		\text{meas} \left\{ \left( \frac{1}{N'} |u| \right)^{p(x)-1} > h \right\}
		&\leq \text{meas} \left\{ |u|^{p^--1} > h \right\} \\
		&\leq \widetilde{C} M^{\frac{N}{N-p^-}} h^{-\frac{N}{N-p^-}} + \widetilde{C} h^{-\frac{Np^-}{(p^--1)(N-p^-)}}.
	\end{align*}
	Rearranging yields
	\begin{equation}\label{?1}
		\text{meas} \left\{ |u| > N' h^{\frac{1}{p(x)-1}} \right\} \leq \widetilde{C} M^{\frac{N}{N-p^-}} h^{-\frac{N}{N-p^-}} + \widetilde{C} h^{-\frac{Np^-}{(p^--1)(N-p^-)}}.
	\end{equation}
	By \eqref{?1}, we obtain
	\begin{align*}
		&\left( \frac{1}{N'} \right)^{p^+-1} \left\| |u|^{p(x)-1} \right\|_{L^{\frac{N}{N-p^-},\infty}(\Omega)} \\
		\leq& \left\| \left( \frac{1}{N'} |u| \right)^{p(x)-1} \right\|_{L^{\frac{N}{N-p^-},\infty}(\Omega)} \\
		=& \sup_{h>0} h \, \text{meas} \left\{ \left( \frac{1}{N'} |u| \right)^{p(x)-1} > h \right\}^{\frac{N-p^-}{N}} \\
		\leq& \sup_{0<h\leq h_0} h \, \text{meas} \left\{ \left( \frac{1}{N'} |u| \right)^{p(x)-1} > h \right\}^{\frac{N-p^-}{N}} 
		+ \sup_{h>h_0} h \, \text{meas} \left\{ \left( \frac{1}{N'} |u| \right)^{p(x)-1} > h \right\}^{\frac{N-p^-}{N}} \\
		\leq& h_0 |\Omega|^{\frac{N-p^-}{N}} + \widetilde{C} M + C h_0^{-\frac{1}{p^--1}}.
	\end{align*}
	Let \( h_0 = |\Omega|^{-\frac{N-p^-}{N}} \). Substituting this into the above, we obtain
	\[
	\left\| |u|^{p(x)-1} \right\|_{L^{\frac{N}{N-p^-},\infty}} \leq \widetilde{C} M + C.
	\]
	Next, we prove 
	\[
	\left\| |\nabla u|^{p(x)-1} \right\|_{L^{N',\infty}} \leq \widetilde{C} M + C.
	\]
	The argument is similar to the previous one.
	\[
	\lambda^{p-} \text{meas}\left\{ |\nabla u| > \lambda \text{ and } |u| < k \right\} \leq \int_{\{ |u| < k \}} |\nabla u|^{p^-} \,\mathrm{d}x
	\]
	\[
	\leq \widetilde{C} \int_{\Omega} |\nabla u|^{p(x)} \,\mathrm{d}x + \widetilde{C} |\Omega|
	\leq \widetilde{C} M k + L,
	\]
	where \( L = L(|\Omega|) \). Let \( \lambda = \mu^{\frac{1}{p^--1}} \). We have
	\[
	\mu^{\frac{p^-}{p^--1}} \text{meas}\left\{ |\nabla u| > \mu^{\frac{1}{p^--1}} \text{ and } |u| \leq k \right\} \leq \widetilde{C} M k + L.
	\]
	There exists \( N' \geq 1 \) such that 
	\[ 
	N' \mu^{\frac{1}{p(x)-1}} > N' \mu^{\frac{1}{p^--1}} .
	\] 
	Combining with the above inequality, we obtain
	\begin{align*}
		\mu^{\frac{p^-}{p^--1}} \text{meas}\left\{ |\nabla u| > N' \mu^{\frac{1}{p(x)-1}} \text{ and } |u| \leq k \right\}
		&\leq \mu^{\frac{p^-}{p^--1}} \text{meas}\left\{ |\nabla u| > \mu^{\frac{1}{p^--1}} \text{ and } |u| \leq k \right\} \\
		&\leq \widetilde{C} M k + L.
	\end{align*}
	Rearranging yields
	\begin{equation}\label{fracfrac}
		\text{meas}\left\{ \left( \frac{1}{N'} |\nabla u|^{p(x)-1} \right) > \mu \text{ and } |u| \leq k \right\} \leq \frac{\widetilde{C} M k + L}{\mu^{\frac{p^-}{p^--1}}}.
	\end{equation}
	Now we estimate the norm of \( \nabla u \). From \eqref{fracfrac}, we have
	\begin{equation}\label{?2}
		\begin{aligned}[b]
			&\operatorname{meas}\Bigl\{ \Bigl( \frac{1}{N'} |\nabla u|^{p(x)-1} \Bigr) > \mu \Bigr\} \\
			&\quad \leq \operatorname{meas}\Bigl\{ \Bigl( \frac{1}{N'} |\nabla u|^{p(x)-1} \Bigr) > \mu \text{ and } |u| < k \Bigr\} \\
			&\qquad + \operatorname{meas}\Bigl\{ \Bigl( \frac{1}{N'} |\nabla u|^{p(x)-1} \Bigr) > \mu \text{ and } |u| > k \Bigr\} \\
			&\quad \leq \frac{\widetilde{C} M k + L}{\mu^{\frac{p^-}{p^--1}}} + \frac{\widetilde{C} (Mk)^{\frac{N}{N-p^-}} + \widetilde{C}}{k^{\frac{Np^-}{N-p^-}}}.
		\end{aligned}
	\end{equation}
	Let \( k = a + b \) with \( a > 0, b > 0 \). Using the inequality \( (x+y)^p \leq 2^p (x^p + y^p) \) for \( p > 0 \), together with \eqref{?2}, we obtain
	\begin{equation}\label{mabl}
		\begin{aligned}[b]
			&\operatorname{meas}\Bigl\{ \Bigl( \frac{1}{N'} |\nabla u| \Bigr)^{p(x)-1} > \mu \Bigr\} \\
			&\quad \leq \frac{\widetilde{C} M a}{\mu^{\frac{p^-}{p^--1}}} + \frac{\widetilde{C} M b}{\mu^{\frac{p^-}{p^--1}}} + \frac{L}{\mu^{\frac{p^-}{p^--1}}} \\
			&\qquad + \widetilde{C} (a+b)^{-\frac{N(1-p^-)}{N-p^-}} M^{\frac{N}{N-p^-}} + \widetilde{C} (a+b)^{-\frac{Np^-}{N-p^-}}.
		\end{aligned}
	\end{equation}
	For \( r > 0, a > 0, b > 0 \), we have
	\[
	\begin{cases}
		(a+b)^{-r} < a^{-r}, \\
		(a+b)^{-r} < b^{-r}.
	\end{cases}
	\]
	Using the above elementary inequality, we rearrange \eqref{mabl} to obtain
	\begin{equation}\label{umabl}
		\begin{aligned}[b]
			&\operatorname{meas}\Bigl\{ \Bigl( \frac{1}{N'} |\nabla u| \Bigr)^{p(x)-1} > \mu \Bigr\} \\
			&\quad \leq \widetilde{C} \left\{ \left[ \frac{Ma}{\mu^{\frac{p^-}{p^--1}}} + a^{\frac{N(1-p^-)}{N-p^-}} M^{\frac{N}{N-p^-}} \right] 
			+ \left[ \frac{Mb}{\mu^{\frac{p^-}{p^--1}}} + b^{-\frac{Np^-}{N-p^-}} \right] + \frac{L}{\mu^{\frac{p^-}{p^--1}}} \right\}.
		\end{aligned}
	\end{equation}
	We now choose suitable \( a \) and \( b \) as follows:
	\begin{equation}
		a = M^{\frac{1}{N-1}} \mu^{\frac{N-p^-}{(N-1)(p^--1)}},
	\end{equation}
	and
	\begin{equation}
		b = M^{\frac{-(N-p^-)}{Np^-+N-p^-}} \mu^{\frac{(N-p^-)p^-}{(p^--1)(Np^-+N-p^-)}}.
	\end{equation}
	Applying Young's inequality to process the above inequality, we deduce that
	\begin{equation}\label{?3}
		\mu \operatorname{meas}\Bigl\{ \Bigl( \frac{1}{N'} |\nabla u| \Bigr)^{p(x)-2} > \mu \Bigr\}^{\frac{N-1}{N}} \leq \widetilde{C} M + \widetilde{C} \mu^{\frac{p^--N}{Np^--N}}.
	\end{equation}
	From \eqref{?3}, we obtain
	\begin{equation}
		\begin{aligned}[b]
			\Bigl( \frac{1}{N'} \Bigr)^{p^+ - 1} \left\| |\nabla u|^{p(x)-1} \right\|_{L^{N',\infty}(\Omega)}
			&\leq \left\| \Bigl( \frac{1}{N'}  |\nabla u|\Bigr)^{p(x)-1} \right\|_{L^{N',\infty}(\Omega)} \\
			&= \sup_{\mu > 0} \mu \, \operatorname{meas}\Bigl\{ \Bigl( \frac{1}{N'}  |\nabla u|\Bigr)^{p(x)-1} > \mu \Bigr\}^{\frac{N-1}{N}} \\
			&= \sup_{0 < \mu \leq \mu_0} \mu \, \operatorname{meas}\Bigl\{ \Bigl( \frac{1}{N'}  |\nabla u|\Bigr)^{p(x)-1} > \mu \Bigr\}^{\frac{N-1}{N}} \\
			&\quad + \sup_{\mu > \mu_0} \mu \, \operatorname{meas}\Bigl\{ \Bigl( \frac{1}{N'}  |\nabla u|\Bigr)^{p(x)-1} > \mu \Bigr\}^{\frac{N-1}{N}} \\
			&\leq \mu_0 |\Omega|^{\frac{N-1}{N}} + \sup_{\mu > \mu_0} \Bigl( C M + C \frac{L^{\frac{p^--1}{p^-}}}{\mu_0^{\frac{p^--N}{Np^--N}}}  \Bigr) \\
			&\leq \widetilde{C} M + C \mu_0^{\frac{p^--N}{Np^--N}} + \widetilde{C} \mu_0 |\Omega|^{\frac{N-1}{N}}.
		\end{aligned}
	\end{equation}
	Choosing \( \mu_0 = \dfrac{1}{|\Omega|^{\frac{N-1}{N}}} \), we conclude that
	\begin{equation}
		\left\| |\nabla u|^{p(x)-1} \right\|_{L^{N',\infty}(\Omega)} \leq \widetilde{C} M + C.
	\end{equation}
	
\end{proof}

\begin{proposition}[A priori estimates for the approximate solutions]
	For any \( k > 0 \), taking \( T_k(u^\varepsilon) \) as a test function in \eqref{appequation} yields
	\begin{equation}
		\int_\Omega A_\varepsilon(x, u^\varepsilon, \nabla u^\varepsilon) \cdot \nabla T_k(u^\varepsilon) \,\mathrm{d}x
		+ \int_\Omega H_\varepsilon(x, u^\varepsilon, \nabla u^\varepsilon) \, T_k(u^\varepsilon) \,\mathrm{d}x
		= \int_\Omega f^\varepsilon \, T_k(u^\varepsilon) \,\mathrm{d}x.
	\end{equation}
	Using the assumptions on \( A \), together with \eqref{as5} and H\"{o}lder's inequality, we deduce from \eqref{nas2} that
	\begin{equation}\label{tpkm}
		\alpha \int_\Omega |\nabla T_k(u^\varepsilon)|^{p(x)} \,\mathrm{d}x \leq k M,
	\end{equation}
	where
	\begin{equation}
		M = \|f^\varepsilon\|_{L^1(\Omega)} + \|b_0\|_{L^{N,1}(\Omega)} \left\| |\nabla u^\varepsilon|^{p(x)-1} \right\|_{L^{N',\infty}(\Omega)} + \|b_1\|_{L^1(\Omega)}.
	\end{equation}
	By Lemma \ref{before}, we have
	\begin{equation}
		\left\| |\nabla u^\varepsilon|^{p(x)-1} \right\|_{L^{N',\infty}(\Omega)} \leq \widetilde{C} M + C.
	\end{equation}
	If \( \|b_0\|_{L^{N,1}(\Omega)} \) is sufficiently small, specifically when \( \widetilde{C} \|b_0\|_{L^{N,1}(\Omega)} < \dfrac{1}{2} \), we directly obtain
	\begin{equation}
		\left\| |\nabla u^\varepsilon|^{p(x)-1} \right\|_{L^{N',\infty}(\Omega)}
		\leq \frac{\widetilde{C} \bigl[ \|f^\varepsilon\|_{L^1(\Omega)} + \|b_1\|_{L^1(\Omega)} \bigr] + C}
		{1 - \widetilde{C} \|b_0\|_{L^{N,1}(\Omega)}},
	\end{equation}
	where \( \widetilde{C} \) and \( C \) are constants depending only on \( N, p^-, p^+, \Omega \).
	
	Consequently, from \eqref{tpkm} we deduce that
	\begin{equation}
		\left\| |\nabla u^\varepsilon|^{p(x)-1} \right\|_{L^{N',\infty}(\Omega)} \leq C_1,
	\end{equation}
	\begin{equation}
		\left\| |u^\varepsilon|^{p(x)-1} \right\|_{L^{\frac{N}{N-p^-},\infty}(\Omega)} \leq C_1,
	\end{equation}
	\begin{equation}\label{tkc1}
		\int_{\Omega} |\nabla T_k(u^{\varepsilon})|^{p(x)} \,\mathrm{d}x \leq C_1,
	\end{equation}
	where \( C_1 \) is a constant depending only on \( \|f^\varepsilon\|_{L^1(\Omega)}, \|b_1\|_{L^1(\Omega)}, \|b_0\|_{L^{N,1}(\Omega)}, N, p^-, p^+, \Omega \).
\end{proposition}
\begin{theorem}
	The sequence \( u^\varepsilon \) converges \(a.e.\) in \( \Omega \)  to a measurable function \( u \).
\end{theorem}

\begin{proof}
	From \eqref{uk}, we have
	\begin{equation}
		\operatorname{meas}\{ |u| > k \} \leq C M^{\frac{Np^-}{N-p^-}} k^{\frac{N(1-p^-)}{N-p^-}} + C k^{-\frac{Np^-}{N-p^-}}.
	\end{equation}
	Consequently,
	\begin{equation}
		\operatorname{meas}\{ |u| > k \} \to 0, \quad \text{as } k \to \infty.
	\end{equation}
	For \( \sigma > 0 \), we define
	\begin{equation}
		E_1 \coloneqq \{ x \in \Omega : |u^\varepsilon| > k \},
	\end{equation}
	\begin{equation}
		E_2 \coloneqq \{ x \in \Omega : |u^\eta| > k \},
	\end{equation}
	\begin{equation}
		E_3 \coloneqq \{ x \in \Omega : |T_k(u^\varepsilon) - T_k(u^\eta)| > \sigma \}.
	\end{equation}
	Noting that \( \{ |u^\varepsilon - u^\eta| > \sigma \} \subset E_1 \cup E_2 \cup E_3 \), we obtain
	\begin{equation}
		\operatorname{meas}\{ |u^\varepsilon - u^\eta| > \sigma \} \leq \operatorname{meas}(E_1) + \operatorname{meas}(E_2) + \operatorname{meas}(E_3).
	\end{equation}
	Given \( \beta > 0 \), choose \( k = k(\beta) \) such that
	\begin{equation}
		\operatorname{meas}(E_1) \leq \frac{\beta}{3},
	\end{equation}
	\begin{equation}
		\operatorname{meas}(E_2) \leq \frac{\beta}{3}.
	\end{equation}
	By \eqref{tkc1} and the compact embedding \cite{Kovacik1991}, we obtain
	\begin{equation}\label{tkdelta}
		T_k(u^\varepsilon) \to \delta_k, \quad \text{strongly in } L^{p(x)}(\Omega).
	\end{equation}
	Hence, \( \{ T_k(u^\varepsilon) \}_n \) is a measurable Cauchy sequence in \( \Omega \), and we have
	\begin{equation}\label{77}
		\operatorname{meas}(E_3) \leq \frac{\beta}{3}.
	\end{equation}
	Combining the above inequalities, we obtain \( \operatorname{meas}\{ |u^\varepsilon - u^\eta| > \sigma \} \leq \beta \). Therefore, \( u^\varepsilon \) converges in measure. By Riesz's theorem, there exist a subsequence of \( u^\varepsilon \) and a measurable function \( u \) such that
	\begin{equation}
		u^\varepsilon \to u, \quad a.e.\text{ in } \Omega.
	\end{equation}
	From \eqref{tkdelta}, we have
	\begin{equation}\label{80}
		T_k(u^\varepsilon) \rightharpoonup T_k(u), \quad \text{weakly in } W_0^{1,p(x)}(\Omega),
	\end{equation}
	\begin{equation}\label{81}
		T_k(u^\varepsilon) \to T_k(u), \quad \text{strongly in } L^{p(x)}(\Omega) \text{ and a.e. in } \Omega,
	\end{equation}
	\begin{equation}\label{82}
		T_k(u^\varepsilon) \stackrel{*}{\rightharpoonup} T_k(u), \quad \text{weakly-\(*\) in } L^{\infty}(\Omega),
	\end{equation}
	\begin{equation}
		\nabla T_k(u^\varepsilon) \rightharpoonup \nabla T_k(u), \quad \text{weakly in } L^{p'(x)}(\Omega).
	\end{equation}
	By the growth assumption \eqref{as5} on \( H \), we have
	\begin{equation}
		\begin{aligned}[b]
			\Bigl| \int_{\Omega} H_\varepsilon(x, u^\varepsilon, \nabla u^\varepsilon) \,\mathrm{d}x \Bigr|
			&\leq \int_{\Omega} |b_0| \, \bigl| \nabla u^\varepsilon \bigr|^{p(x)-1} \,\mathrm{d}x + \int_{\Omega} |b_1| \,\mathrm{d}x \\
			&\leq \|b_0\|_{L^{N,1}(\Omega)} \left\| |\nabla u^\varepsilon|^{p(x)-1} \right\|_{L^{N',\infty}(\Omega)} + \|b_1\|_{L^1(\Omega)} \\
			&\leq C.
		\end{aligned}
	\end{equation}
	Hence, there exists \( \Lambda \in L^1(\Omega) \) such that
	\begin{equation}\label{weakcvg}
		H_\varepsilon(x, u^\varepsilon, \nabla u^\varepsilon) \rightharpoonup \Lambda, \quad \text{weakly in } L^1(\Omega).
	\end{equation}
\end{proof}

\begin{remark}\label{aclremark}
	For \( 0 < \varepsilon < \dfrac{1}{k} \), note that
	\begin{equation}
		A\bigl(x, T_{\frac{1}{\varepsilon}}(u^\varepsilon), \nabla T_k(u^\varepsilon)\bigr) = A\bigl(x, T_k(u^\varepsilon), \nabla T_k(u^\varepsilon)\bigr), \quad a.e. \text{ in } \Omega,
	\end{equation}
	and
	\begin{equation}
		\bigl| A\bigl(x, T_k(u^\varepsilon), \nabla T_k(u^\varepsilon)\bigr) \bigr| \leq C \bigl[ L(x) + \bigl| \nabla T_k(u^\varepsilon) \bigr|^{p(x)-1} \bigr].
	\end{equation}
	For every \( k > 0 \), we have
	\begin{equation}
		A_\varepsilon\bigl(x, T_k(u^\varepsilon), \nabla T_k(u^\varepsilon)\bigr) \rightharpoonup \varphi_k, \quad \text{weakly in } L^{p'(x)}(\Omega).
	\end{equation}
\end{remark}
\begin{lemma}\label{an0}
	The sequence \( u^\varepsilon \) defined in \eqref{appequation} satisfies
	\begin{equation}
		\lim_{n \to \infty} \limsup_{\varepsilon \to 0} \frac{1}{n} \int_{|u^\varepsilon| \leq n} A_\varepsilon(x, u^\varepsilon, \nabla u^\varepsilon) \cdot \nabla u^\varepsilon \,\mathrm{d}x = 0.
	\end{equation}
\end{lemma}

\begin{proof}
	Taking \( \dfrac{T_n(u^\varepsilon)}{n} \) as a test function in equation \eqref{appequation}, we obtain
	\begin{equation}\label{an01}
		\begin{aligned}[b]
			&\frac{1}{n} \int_{\{|u^\varepsilon| \leq n\}} A_\varepsilon(x, u^\varepsilon, \nabla u^\varepsilon) \cdot \nabla T_n(u^\varepsilon) \,\mathrm{d}x \\
			\leq& \frac{1}{n} \int_\Omega |f^\varepsilon| \, |T_n(u^\varepsilon)| \,\mathrm{d}x 
			+ \frac{1}{n} \int_\Omega |H_\varepsilon(x, u^\varepsilon, \nabla u^\varepsilon) \, T_n(u^\varepsilon)| \,\mathrm{d}x.
		\end{aligned}
	\end{equation}
	By \eqref{weakcvg}, we have
	\begin{equation}\label{an02}
		\lim_{\varepsilon \to 0} \int_\Omega H_\varepsilon(x, u^\varepsilon, \nabla u^\varepsilon) \, T_n(u^\varepsilon) \,\mathrm{d}x = \int_\Omega \Lambda \, T_n(u) \,\mathrm{d}x.
	\end{equation}
	Applying the Lebesgue dominated convergence theorem and passing to the limit as \( n \to \infty \), we deduce that
	\begin{equation}\label{an03}
		\lim_{n \to \infty} \lim_{\varepsilon \to 0} \frac{1}{n} \int_\Omega H_\varepsilon(x, u^\varepsilon, \nabla u^\varepsilon) \, T_n(u^\varepsilon) \,\mathrm{d}x 
		= \lim_{n \to \infty} \frac{1}{n} \int_\Omega \Lambda \, T_n(u) \,\mathrm{d}x = 0,
	\end{equation}
	and
	\begin{equation}\label{an04}
		\lim_{n \to \infty} \lim_{\varepsilon \to 0} \frac{1}{n} \int_\Omega |f^\varepsilon| \, |T_n(u^\varepsilon)| \,\mathrm{d}x 
		= \lim_{n \to \infty} \frac{1}{n} \int_\Omega |f| \, |T_n(u)| \,\mathrm{d}x = 0.
	\end{equation}
	Combining \eqref{an01}-\eqref{an04}, we obtain
	\begin{equation}\label{an05}
		\lim_{n \to \infty} \limsup_{\varepsilon \to 0} \frac{1}{n} \int_{|u^\varepsilon| \leq n} A_\varepsilon(x, u^\varepsilon, \nabla u^\varepsilon) \cdot \nabla u^\varepsilon \,\mathrm{d}x = 0,
	\end{equation}
	which completes the proof of the lemma.
\end{proof}

\section{Estimates in the General Case}
In the general case where \eqref{tpkm} does not hold, we employ the technique developed by Bottaro and Marina (see \cite{Guibe2008}) for studying linear problems with right-hand side data in the dual space. We adapt this approach here to a nonlinear problem with measurable right-hand side and with the coefficient \( b_0 \) belonging to a Lorentz space. The idea is, in a certain sense, to decompose \( b_0 \) into a finite sum of terms, each of which satisfies \eqref{tpkm}. Since \( |\Omega| \) is finite, the sets
\[
\bigl\{ x \in \Omega : |u^\varepsilon(x)| = c \bigr\} > 0
\]
having positive measure are at most countable. Let \( Z_\varepsilon^c \) be the countable union of all such sets. Consequently, its complement \( Z_\varepsilon = \Omega \setminus Z_\varepsilon^c \) is the union of sets on which
\[
\bigl\{ x \in \Omega : |u^\varepsilon(x)| = c \bigr\} = 0
\]
for every \( c \). We have
\[
\nabla u^\varepsilon = 0 \quad a.e. \text{ on } \bigl\{ x \in \Omega : |u^\varepsilon(x)| = c \bigr\}.
\]
Since \( Z_\varepsilon^c \) is at most a countable union, we have
\begin{equation}\label{uaez}
	\nabla u^\varepsilon = 0 \quad a.e. \text{ on }  Z_\varepsilon^c .
\end{equation}
In the subsequent part of the proof, we will consider the measure of the sets
\[
\left\{ Z_\varepsilon \cup \{ m_{i+1} < |u^\varepsilon| < m_i \} \right\}.
\]
The parameters \( m_i, m_{i+1} \) are chosen conveniently. Since there are at most countably many constants \( c \) for which the set \( \{ |u^\varepsilon(x)| = c \} \) has strictly positive measure, by working on \( Z_\varepsilon \) we eliminate this possibility. For a function of \( m_i \) (\( 0 < m < m_i \)), the map
\[
m \mapsto \bigl| Z_\varepsilon \cup \bigl\{ m_{i+1} < |u^\varepsilon| < m_i \bigr\} \bigr|
\]
is continuous.

Using the decreasing rearrangement of \( b_0 \) and the restriction \( b_0|_{Z_\varepsilon \cap E} \) on \( Z_\varepsilon \cap E \), we have
\[
f^*(t) = \inf \bigl\{ \alpha \mid \lambda_f(\alpha) \leq t \bigr\}, \quad \text{and} \quad \lambda_f(\alpha) = \bigl| \bigl\{ x \in \Omega \mid f(x) > \alpha \bigr\} \bigr|.
\]
Consequently,
\[
(b_0|_{Z_\varepsilon \cap E})^*(t) \leq b_0^*(t), \quad t \in [0, |Z_\varepsilon \cap E|].
\]
For any measurable set \( E \), we have
\begin{equation}
	\begin{aligned}[b]
		\|b_0\|_{L^{N,1}(Z_\varepsilon \cap \{|u^\varepsilon| > m\})}
		&= \int_0^{|Z_\varepsilon \cap \{|u^\varepsilon| > m\}|} \bigl( b_0|_{Z_\varepsilon \cap \{|u^\varepsilon| > m\}} \bigr)^* (t) \cdot t^{-\frac{1}{N}} \,\mathrm{d}t \\
		&\leq \int_0^{|Z_\varepsilon \cap \{|u^\varepsilon| > m\}|} b_0^*(t) \cdot t^{-\frac{1}{N}} \,\mathrm{d}t.
	\end{aligned}
\end{equation}
In the following, we work under the condition
\begin{equation}\label{cnp12}
	C(N, p) \dfrac{p'}{\alpha} \|b_0\|_{L^{N,1}(Z_\varepsilon)}
	= C(N, p) \dfrac{p'}{\alpha} \int_0^{|Z_\varepsilon|} b_0^*(t) \cdot t^{-\frac{1}{N}} \,\mathrm{d}t
	\leq \dfrac{1}{2}.
\end{equation}
\begin{theorem}
	There exists \( \delta > 0 \), independent of \( \varepsilon \), defined by
	\begin{equation}
		C(N, p) \frac{p'}{\alpha} \int_0^\delta b_0^*(t) \cdot t^{-\frac{1}{N}} \,\mathrm{d}t = \frac{1}{2},
	\end{equation}
	such that
	\begin{equation}
		\bigl| Z_\varepsilon \cap \{ |u^\varepsilon| > m_1 \} \bigr| = \delta.
	\end{equation}
\end{theorem}
\begin{proof}
	
\subsection*{General case: Step 1}
	
	We first choose \( m = m_1 = 0 \). If the equality above does not hold, we select \( m = m_1 > 0 \) such that
	\begin{equation}\label{lm1}
		C(N, p) \frac{p'}{\alpha} \int_0^{|Z_\varepsilon \cap \{|u^\varepsilon| > m_1\}|} b_0^*(t) \cdot t^{\frac{1}{N}} \frac{\mathrm{d}t}{t} = \frac{1}{2}.
	\end{equation}
	Indeed, the function \( m \mapsto \bigl| Z_\varepsilon \cap \{ |u^\varepsilon| > m \} \bigr| \) is continuous and decreasing, and tends to \( 0 \) as \( m \to \infty \). Once this is achieved, we take \( T_k(S_m(u^\varepsilon)) \) as a test function, where \( S_m(s) = s - T_m(s) \) for all \( s \in \mathbb{R} \). We obtain
	\begin{equation}\label{lm2}
		\begin{aligned}[b]
			&\int_{\Omega} A_\varepsilon(x, u^\varepsilon, \nabla u^\varepsilon) \cdot \nabla T_k(S_m(u^\varepsilon)) \,\mathrm{d}x \\
			=& -\int_{\Omega} H_\varepsilon(x, u^\varepsilon, \nabla u^\varepsilon) \, T_k(S_m(u^\varepsilon)) \,\mathrm{d}x 
			+ \int_{\Omega} f^\varepsilon \, T_k(S_m(u^\varepsilon)) \,\mathrm{d}x.
		\end{aligned}
	\end{equation}
	Consequently,
	\begin{equation}\label{lm3}
		\int_{\Omega} A_\varepsilon(x, u^\varepsilon, \nabla u^\varepsilon) \cdot \nabla T_k(S_m(u^\varepsilon)) \,\mathrm{d}x
		\geq \alpha \int_{\Omega} \bigl| \nabla T_k(S_m(u^\varepsilon)) \bigr|^{p(x)} \,\mathrm{d}x,
	\end{equation}
	and
	\begin{equation}\label{lm4}
		\int_{\Omega} f^\varepsilon \, T_k(S_m(u^\varepsilon)) \,\mathrm{d}x \leq k \|f^\varepsilon\|_{L^1(\Omega)}.
	\end{equation}
	By the growth assumption on \( H \), we have
	\begin{equation}
		\begin{aligned}[b]
			&\left|\int_{\Omega} H_\varepsilon(x, u^\varepsilon, \nabla u^\varepsilon) \, T_k(S_m(u^\varepsilon)) \,\mathrm{d}x\right| \\
			\leq& k \int_{\{|u^\varepsilon| > m\}} |H_\varepsilon(x, u^\varepsilon, \nabla u^\varepsilon)| \,\mathrm{d}x \\
			\leq& k \left[ \int_{\{|u^\varepsilon| > m\}} b_0 |\nabla u^\varepsilon|^{p(x)-1} \,\mathrm{d}x + \int_{\Omega} b_1 \,\mathrm{d}x \right] \\
			=& k \left[ \int_{Z_\varepsilon \cap \{|u^\varepsilon| > m\}} b_0 |\nabla S_m(u^\varepsilon)|^{p(x)-1} \,\mathrm{d}x + \int_{\Omega} b_1 \,\mathrm{d}x \right] \\
			\leq& k \left[ \|b_0\|_{L^{N,1}(Z_\varepsilon \cap \{|u^\varepsilon| > m\})} \left\| |\nabla S_m(u^\varepsilon)|^{p(x)-1} \right\|_{L^{N',\infty}(\Omega)} + \|b_1\|_{L^1(\Omega)} \right].
		\end{aligned}
	\end{equation}
	Combining \eqref{lm1}-\eqref{lm4}, we obtain
	\begin{equation}\label{m1k}
		\int_{\Omega} |\nabla T_k(S_m(u^\varepsilon))|^{p(x)} \,\mathrm{d}x \leq M_1 k,
	\end{equation}
	where
	\begin{equation}
		M_1 = \frac{1}{\alpha} \Bigl[ \|f^\varepsilon\|_{L^1(\Omega)} + \|b_0\|_{L^{N,1}(Z_\varepsilon \cap \{|u^\varepsilon| > m\})} \left\| |\nabla S_m(u^\varepsilon)|^{p(x)-1} \right\|_{L^{N',\infty}(\Omega)} + \|b_1\|_{L^1(\Omega)} \Bigr].
	\end{equation}
	By Lemma \ref{before}, we deduce that
	\begin{equation}
		\begin{aligned}
			&\left\| |\nabla S_m(u^\varepsilon)|^{p(x)-1} \right\|_{L^{N',\infty}(\Omega)} \\
			\leq& \frac{C(N, p^-)}{\alpha} \Bigl[ \|b_0\|_{L^{N,1}(Z_\varepsilon \cap \{|u^\varepsilon| > m\})} \left\| |\nabla S_m(u^\varepsilon)|^{p(x)-1} \right\|_{L^{N',\infty}(\Omega)} 
			\quad + \|b_1\|_{L^1(\Omega)} + \|f^\varepsilon\|_{L^1(\Omega)} \Bigr] + C.
		\end{aligned}
	\end{equation}
	Taking \( m = m_1 \) in the above, we obtain
	\begin{equation}
		\left\| |\nabla S_{m_1}(u^{\varepsilon})|^{p(x)-1} \right\|_{L^{N',\infty}(\Omega)} \leq 2C \bigl[ \|b_1\|_{L^1(\Omega)} + \|f^{\varepsilon}\|_{L^1(\Omega)} \bigr] + C.
	\end{equation}
\subsection*{General case: Step 2}
	Define the function \( S_{m,m_1} : \mathbb{R} \to \mathbb{R} \), \( 0 \leq m < m_1 \), as follows:
	\begin{equation}\label{smms}
		S_{m,m_1}(s) =
		\begin{cases}
			m_1 - m, & s > m_1, \\
			s - m, & m \leq s \leq m_1, \\
			0, & -m \leq s \leq m, \\
			s + m, & -m_1 \leq s \leq -m, \\
			m - m_1, & s < -m_1.
		\end{cases}
	\end{equation}
	Observe that setting \( m_0 = \infty \), the function \( S_m \) defined below coincides with \( S_{m,m_0}(s) \):
	\begin{equation}
		S_m(s) =
		\begin{cases}
			0, & |s| \leq m, \\
			(|s| - m) \operatorname{sign}(s), & |s| > m.
		\end{cases}
	\end{equation}
	Taking \( T_k(S_{m,m_1}(u^\varepsilon)) \) as a test function, we obtain
	\begin{equation}\label{smmtest}
		\begin{aligned}[b]
			&\int_{\Omega} A_\varepsilon(x, u^\varepsilon, \nabla u^\varepsilon) \cdot \nabla T_k(S_{m,m_1}(u^\varepsilon)) \,\mathrm{d}x \\
			=& -\int_{\Omega} H_\varepsilon(x, u^\varepsilon, \nabla u^\varepsilon) \, T_k(S_{m,m_1}(u^\varepsilon)) \,\mathrm{d}x 
			+ \int_{\Omega} f^\varepsilon \, T_k(S_{m,m_1}(u^\varepsilon)) \,\mathrm{d}x.
		\end{aligned}
	\end{equation}
	Using the assumptions on \( A \) and \( f \), we have
	\begin{equation}\label{atsts}
		\int_{\Omega} A_\varepsilon(x, u^\varepsilon, \nabla u^\varepsilon) \cdot \nabla T_k(S_{m,m_1}(u^\varepsilon)) \,\mathrm{d}x
		\geq \alpha \int_{\Omega} |\nabla T_k(S_{m,m_1}(u^\varepsilon))|^{p(x)} \,\mathrm{d}x,
	\end{equation}
	and
	\begin{equation}\label{ftskf}
		\int_{\Omega} f^\varepsilon \, T_k(S_{m,m_1}(u^\varepsilon)) \,\mathrm{d}x \leq k \|f^\varepsilon\|_{L^1(\Omega)}.
	\end{equation}
	Moreover, since \( S_{m,m_1}(s) = 0 \) for \( |s| \leq m \), by the growth assumption on \( H \) we obtain
	\begin{equation}\label{hsmm}
		\begin{aligned}[b]
			&\int_{\Omega} H_\varepsilon(x, u^\varepsilon, \nabla u^\varepsilon) \, T_k(S_{m,m_1}(u^\varepsilon)) \,\mathrm{d}x \\
			\leq& k \int_{|u^\varepsilon| > m} b_0 |\nabla u^\varepsilon|^{p(x)-1} \,\mathrm{d}x + \int_{\Omega} b_1 \,\mathrm{d}x \\
			\leq& k \int_{m < |u^\varepsilon| < m_1} b_0 |\nabla u^\varepsilon|^{p(x)-1} \,\mathrm{d}x 
			+ \int_{|u^\varepsilon| \geq m_1} b_0 |\nabla u^\varepsilon|^{p(x)-1} \,\mathrm{d}x + \int_{\Omega} b_1 \,\mathrm{d}x.
		\end{aligned}
	\end{equation}
	Using the properties of \( Z_\varepsilon \) and the Hölder inequality in Lorentz spaces, we estimate the right-hand side of \eqref{atsts} as follows:
	\begin{equation}\label{bbb}
		\begin{aligned}[b]
			&\int_{\{m<|u^\varepsilon|<m_1\}} b_0 |\nabla u^\varepsilon|^{p(x)-1} \,\mathrm{d}x \\
			=& \int_{Z_\varepsilon \cap \{m<|u^\varepsilon|<m_1\}} b_0 |\nabla S_{m,m_1}(u^\varepsilon)|^{p(x)-1} \,\mathrm{d}x \\
			\leq& \|b_0\|_{L^{N,1}(Z_\varepsilon \cap \{m<|u^\varepsilon|<m_1\})}  \left\| |\nabla S_{m,m_1}(u^\varepsilon)|^{p(x)-1} \right\|_{L^{N',\infty}(\Omega)} .
		\end{aligned}
	\end{equation}
	Similarly, for the second term on the right-hand side of \eqref{ftskf}, we have
	\begin{equation}
		\begin{aligned}[b]
			&\int_{\{|u^\varepsilon|\geq m_1\}} b_0 |\nabla u^\varepsilon|^{p(x)-1} \,\mathrm{d}x \\
			=& \int_{\Omega} b_0 |\nabla S_m(u^\varepsilon)|^{p(x)-1} \,\mathrm{d}x \\
			\leq& \|b_0\|_{L^{N,1}(\Omega)}  \left\| |\nabla S_m(u^\varepsilon)|^{p(x)-1} \right\|_{L^{N',\infty}(\Omega)} .
		\end{aligned}
	\end{equation}
	Therefore,
	\begin{equation}\label{long}
		\begin{aligned}[b]
			&\Bigl| \int_{\Omega} H_\varepsilon(x, u^\varepsilon, \nabla u^\varepsilon) \, T_k(S_{m,m_1}(u^\varepsilon)) \,\mathrm{d}x \Bigr| \\
			\leq k &\Bigl[ \|b_0\|_{L^{N,1}(Z_\varepsilon \cap \{m<|u^\varepsilon|<m_1\})}  \left\| |\nabla S_{m,m_1}(u^\varepsilon)|^{p(x)-1} \right\|_{L^{N',\infty}(\Omega)}  \\
			+& \|b_0\|_{L^{N,1}(\Omega)}  \left\| |\nabla S_{m_1}(u^\varepsilon)|^{p(x)-1} \right\|_{L^{N',\infty}(\Omega)} + \|b_1\|_{L^1(\Omega)} \Bigr].
		\end{aligned}
	\end{equation}
	Combining \eqref{smms}-\eqref{bbb}, for every \( k > 0 \) we obtain
	\begin{equation}
		\int_{\Omega} |\nabla T_k(S_{m,m_1}(u^\varepsilon))|^{p(x)} \,\mathrm{d}x \leq M_2 k,
	\end{equation}
	where
	\begin{equation}
		\begin{aligned}[b]
			M_2 &= \frac{1}{\alpha} \Bigl[ \|b_0\|_{L^{N,1}(Z_\varepsilon \cap \{m < |u^\varepsilon| < m_1\})} \left\| |\nabla S_{m,m_1}(u^\varepsilon)|^{p(x)-1} \right\|_{L^{N',\infty}(\Omega)} \\
			&\quad + \|b_0\|_{L^{N,1}(\Omega)} \left\| |\nabla S_{m_1}(u^\varepsilon)|^{p(x)-1} \right\|_{L^{N',\infty}(\Omega)}  + \|b_1\|_{L^1(\Omega)} + \|f^\varepsilon\|_{L^1(\Omega)} \Bigr].
		\end{aligned}
	\end{equation}
	By Lemma \ref{before}, we have
	\begin{equation}
		\left\| |\nabla S_{m,m_1}(u^\varepsilon)|^{p(x)-1} \right\|_{L^{N',\infty}(\Omega)} \leq C(N, p^-) M_2 + C.
	\end{equation}
	Using \eqref{uaez}, we obtain
	\begin{equation}\label{b0b0b0}
		\begin{aligned}[b]
			&\|b_0\|_{L^{N,1}(Z_\varepsilon \cap \{m < |u^\varepsilon| < m_1\})} \\
			=& \int_0^{|Z_\varepsilon \cap \{m < |u^\varepsilon| < m_1\}|} \bigl( b_0|_{Z_\varepsilon \cap \{m < |u^\varepsilon| < m_1\}} \bigr)^* (t) \cdot t^{\frac{1}{N}} \frac{\mathrm{d}t}{t} \\
			\leq& \int_0^{|Z_\varepsilon \cap \{m < |u^\varepsilon| < m_1\}|} b_0^*(t) \cdot t^{\frac{1}{N}} \frac{\mathrm{d}t}{t}.
		\end{aligned}
	\end{equation}
	In this situation, we have
	\begin{equation}\label{mm212}
		C(N, p^-) \frac{1}{\alpha} \int_0^{|Z_\varepsilon \cap \{m < |u^\varepsilon| < m_1\}|} b_0^*(t) \cdot t^{\frac{1}{N}} \frac{\mathrm{d}t}{t} \leq \frac{1}{2}.
	\end{equation}
	We first choose \( m = m_2 = 0 \). If the above inequality does not hold, we select \( m = m_2 > 0 \) such that
	\begin{equation}
		C(N, p^-) \frac{1}{\alpha} \int_0^{|Z_\varepsilon \cap \{m_2 < |u^\varepsilon| < m_1\}|} b_0^*(t) \cdot t^{\frac{1}{N}} \frac{\mathrm{d}t}{t} = \frac{1}{2}.
	\end{equation}
	Indeed, the function \( m \mapsto \bigl| Z_\varepsilon \cap \{m < |u^\varepsilon| < m_1\} \bigr| \) is continuous and decreasing, and as \( m \to m_1 \) it tends to \( 0 \), while as \( m \to 0 \) it tends to \( \bigl| Z_\varepsilon \cap \{0 < |u^\varepsilon| < m_1\} \bigr| \).
	
\end{proof}
\begin{remark}
	Note that \( m_2 \) depends on \( \varepsilon \), and we have
	\begin{equation}
		\bigl| Z_{\varepsilon} \cap \{ m_2 < |u^{\varepsilon}| < m_1 \} \bigr| = \delta.
	\end{equation}
	Choosing \( m = m_2 \), from \eqref{long} we obtain
	\begin{equation}
		\begin{aligned}
			&\left\| |\nabla S_{m_2, m_1}(u^{\varepsilon})|^{p(x)-1} \right\|_{L^{N',\infty}(\Omega)} \\
			\leq& \frac{2C(N, p^-)}{\alpha} \Bigl[ \|b_0\|_{L^{N,1}(\Omega)} \left\| |\nabla S_{m_1}(u^{\varepsilon})|^{p(x)-1} \right\|_{L^{N',\infty}(\Omega)} 
			+ \|b_1\|_{L^1(\Omega)} + \|f^{\varepsilon}\|_{L^1(\Omega)} \Bigr].
		\end{aligned}
	\end{equation}
\end{remark}
	Define the function \( S_{m,m_2} : \mathbb{R} \to \mathbb{R} \), \( 0 \leq m < m_2 \), as follows:
	\begin{equation}
		S_{m,m_2}(s) =
		\begin{cases}
			m_2 - m, & s > m_2, \\
			s - m, & m \leq s \leq m_2, \\
			0, & -m \leq s \leq m, \\
			s + m, & -m_2 \leq s \leq -m, \\
			m - m_2, & s < -m_2.
		\end{cases}
	\end{equation}
	Taking \( T_k(S_{m,m_2}(u^\varepsilon)) \) as a test function, we obtain
	\begin{equation}\label{testm2}
		\begin{aligned}[b]
			&\int_{\Omega} A_\varepsilon(x, u^\varepsilon, \nabla u^\varepsilon) \cdot \nabla T_k(S_{m,m_2}(u^\varepsilon)) \,\mathrm{d}x \\
			=& -\int_{\Omega} H_\varepsilon(x, u^\varepsilon, \nabla u^\varepsilon) \, T_k(S_{m,m_2}(u^\varepsilon)) \,\mathrm{d}x 
			+ \int_{\Omega} f^\varepsilon \, T_k(S_{m,m_2}(u^\varepsilon)) \,\mathrm{d}x.
		\end{aligned}
	\end{equation}
	Proceeding similarly to the previous estimates, we bound the terms in \eqref{testm2} as follows:
	\begin{equation}
		\begin{aligned}[b]
			&\Bigl| \int_{\Omega} H_\varepsilon(x, u^\varepsilon, \nabla u^\varepsilon) \, T_k(S_{m,m_2}(u^\varepsilon)) \,\mathrm{d}x \Bigr| \\
			\leq& k \Bigl[ \int_{\{m < |u^\varepsilon| < m_2\}} b_0 |\nabla u^\varepsilon|^{p(x)-1} \,\mathrm{d}x
			+ \int_{\{m_2 < |u^\varepsilon| < m_1\}} b_0 |\nabla u^\varepsilon|^{p(x)-1} \,\mathrm{d}x \\
			&+ \int_{\{|u^\varepsilon| > m\}} b_0 |\nabla u^\varepsilon|^{p(x)-1} \,\mathrm{d}x + \int_{\Omega} b_1 \,\mathrm{d}x \Bigr] \\
			\leq& k \Bigl[ \|b_0\|_{L^{N,1}(Z_\varepsilon \cap \{0 < |u^\varepsilon| < m_2\})}  \bigl\| |\nabla S_{m,m_2}(u^\varepsilon)|^{p(x)-1} \bigr\|_{L^{N',\infty}(\Omega)}  + \|b_1\|_{L^1(\Omega)} \\
			&+ \|b_0\|_{L^{N,1}(\Omega)}  \bigl\| |\nabla S_{m_2,m_1}(u^\varepsilon)|^{p(x)-1} \bigr\|_{L^{N',\infty}(\Omega)}^{p}  \\
			&+ \|b_0\|_{L^{N,1}(\Omega)}  \bigl\| |\nabla S_{m_1}(u^\varepsilon)|^{p(x)-1} \bigr\|_{L^{N',\infty}(\Omega)}^{p}  \Bigr].
		\end{aligned}
	\end{equation}
	From this, we deduce that
	\begin{equation}
		\left\| |\nabla S_{m_2, m_1}(u^\varepsilon)|^{p(x)-1} \right\|_{L^{N',\infty}(\Omega)}\leq C(N, p^-) M_3 k,
	\end{equation}
	where
	\begin{equation}
		\begin{aligned}[b]
			M_3 &= \frac{1}{\alpha} \Bigl[ \|b_0\|_{L^{N,1}(Z_\varepsilon \cap \{m < |u^\varepsilon| < m_2\})} \left\| |\nabla S_{m, m_2}(u^\varepsilon)|^{p(x)-1} \right\|_{L^{N',\infty}(\Omega)} + \|b_1\|_{L^1(\Omega)} + \|f^{\varepsilon}\|_{L^1(\Omega)} \\
			&+ \|b_0\|_{L^{N,1}(\Omega)} \left\| |\nabla S_{m_2, m_1}(u^\varepsilon)|^{p(x)-1} \right\|_{L^{N',\infty}(\Omega)}
			+ \|b_0\|_{L^{N,1}(\Omega)} \left\| |\nabla S_{m_1}(u^\varepsilon)|^{p(x)-1} \right\|_{L^{N',\infty}(\Omega)} \Bigr].
		\end{aligned}
	\end{equation}
	By \eqref{uaez}, we have
	\begin{equation}\label{b0b0t}
		\|b_0\|_{L^{N,1}(Z_\varepsilon \cap \{m < |u^\varepsilon| < m_2\})}
		\leq \int_0^{|Z_\varepsilon \cap \{m < |u^\varepsilon| < m_2\}|} b_0^*(t) \cdot t^{\frac{1}{N}} \frac{\mathrm{d}t}{t}.
	\end{equation}
	In this situation, we have
	\begin{equation}\label{0m212}
		C(N, p^-) \frac{1}{\alpha} \int_0^{|Z_\varepsilon \cap \{0 < |u^\varepsilon| < m_2\}|} b_0^*(t) \cdot t^{\frac{1}{N}} \frac{\mathrm{d}t}{t} \leq \frac{1}{2}.
	\end{equation}
	We first choose \( m = m_3 = 0 \). If the equality does not hold, we select \( m = m_3 > 0 \) such that
	\begin{equation}
		C(N, p^-) \frac{1}{\alpha} \int_0^{|Z_\varepsilon \cap \{m_3 < |u^\varepsilon| < m_1\}|} b_0^*(t) \cdot t^{\frac{1}{N}} \frac{\mathrm{d}t}{t} = \frac{1}{2}.
	\end{equation}
	Note that \( m_3 \) depends on \( \varepsilon \), and we have
	\begin{equation}\label{m2m3}
		\bigl| Z_\varepsilon \cap \{ m_3 < |u^\varepsilon| < m_2 \} \bigr| = \delta.
	\end{equation}
	Choosing \( m = m_3 \), from \eqref{testm2} we obtain
	\begin{equation}\label{sm2m3}
		\begin{aligned}[b]
			&\left\| |\nabla S_{m_3, m_2}(u^\varepsilon)|^{p(x)-1} \right\|_{L^{N',\infty}(\Omega)}\\
			\leq& \frac{2C(N, p^-)}{\alpha} \Bigl[ \|b_0\|_{L^{N,1}(\Omega)} \left\| |\nabla S_{m_2, m_1}(u^\varepsilon)|^{p(x)-1} \right\|_{L^{N',\infty}(\Omega)} \\
			&\quad + \|b_0\|_{L^{N,1}(\Omega)} \left\| |\nabla S_{m_1}(u^\varepsilon)|^{p(x)-1} \right\|_{L^{N',\infty}(\Omega)}
			+ \|b_1\|_{L^1(\Omega)} + \|f^\varepsilon\|_{L^1(\Omega)} \Bigr] .
		\end{aligned}
	\end{equation}
	We repeat this process until it terminates. At some index \( i = I \) (depending on \( \varepsilon \)), we have
	\begin{equation}
		C(N, p^-) \frac{1}{\alpha} \int_0^{|Z_\varepsilon \cap \{0 < |u^\varepsilon| < m_{I-1}\}|} b_0^*(t) \cdot t^{\frac{1}{N}} \frac{\mathrm{d}t}{t} \leq \frac{1}{2}.
	\end{equation}
	Then we choose \( m_I = 0 \). To estimate \( I \), we note that
	\begin{equation}
		\begin{aligned}[b]
			|\Omega| \geq |Z_\varepsilon|
			&\geq \bigl| Z_\varepsilon \cap \{ |u^\varepsilon| > m_1 \} \bigr| + \bigl| Z_\varepsilon \cap \{ m_2 < |u^\varepsilon| < m_1 \} \bigr| + \cdots \\
			&\quad + \bigl| Z_\varepsilon \cap \{ m_{I-1} < |u^\varepsilon| < m_{I-2} \} \bigr|.
		\end{aligned}
	\end{equation}
	Combining \eqref{cnp12}, \eqref{b0b0b0} and \eqref{b0b0t}, we obtain
	\begin{equation}
		\bigl| Z_\varepsilon \cap \{ |u^\varepsilon| > m_1 \} \bigr|
		= \bigl| Z_\varepsilon \cap \{ m_2 < |u^\varepsilon| < m_1 \} \bigr|
		= \cdots
		= \bigl| Z_\varepsilon \cap \{ m_{I-1} < |u^\varepsilon| < m_{I-2} \} \bigr|
		= \delta.
	\end{equation}
	Since \( (I-1)\delta \leq |\Omega| \), we deduce that
	\begin{equation}\label{ii}
		I \leq I^*,
	\end{equation}
	where
	\begin{equation}
		I^* = 1 + \left[ \frac{|\Omega|}{\delta} \right],
	\end{equation}
	and define
	\[ 
	[s] \coloneqq \inf \{ n \in \mathbb{N} : s \leq n \}. 
	\]
	Define \( S_{m_1, m_0} = S_{m_1} \) with \( m_0 = +\infty \), and set
	\begin{equation}
		\begin{cases}
			x_i = \left\| |\nabla S_{m_i, m_{i-1}}(u^\varepsilon)|^{p(x)-1} \right\|_{L^{N',\infty}(\Omega)}, & \text{for } 1 \leq i \leq I, \\[4pt]
			a = 2C(N, p^-) \|b_0\|_{L^{N,1}(\Omega)}, \\[4pt]
			b = 2C(N, p^-) \bigl[ \|b_1\|_{L^1(\Omega)} + \|f^\varepsilon\|_{L^1(\Omega)} \bigr].
		\end{cases}
	\end{equation}
	Observe that
	\begin{equation}
		x_1 = \left\| |\nabla S_{m_1, m_0}(u^\varepsilon)|^{p(x)-1} \right\|_{L^{N',\infty}(\Omega)}
		= \left\| |\nabla S_{m_1}(u^\varepsilon)|^{p(x)-1} \right\|_{L^{N',\infty}(\Omega)}.
	\end{equation}
	From \eqref{m1k}, \eqref{mm212}, \eqref{0m212} and \eqref{m2m3}, we obtain
	\begin{equation}
		\begin{aligned}
			x_1 &\leq b, \\
			x_2 &\leq a x_1 + b, \\
			x_3 &\leq a x_2 + a x_1 + b, \\
			&\ \vdots \\
			x_I &\leq a x_{I-1} + a x_{I-2} + \cdots + a x_1 + b, \quad I \leq I^*.
		\end{aligned}
	\end{equation}
	By induction, one can prove that
	\begin{equation}
		x_i \leq (a+1)^{i-1} b, \quad \text{for } 1 \leq i \leq I.
	\end{equation}
	Since \( m_I = 0 \), we have
	\begin{equation}
		|\nabla u^{\varepsilon}|^{p(x)-1}
		= \sum_{i=1}^I |\nabla u^{\varepsilon}|^{p(x)-1} \chi_{\{ m_i < |u^{\varepsilon}| < m_{i-1} \}}
		= \sum_{i=1}^I |\nabla S_{m_i, m_{i-1}}(u^{\varepsilon})|^{p(x)-1}.
	\end{equation}
	Therefore, from \eqref{sm2m3} and \eqref{ii}, we deduce that
	\begin{equation}
		\begin{aligned}
			\bigl\| |\nabla u^{\varepsilon}|^{p(x)-1} \bigr\|_{L^{N',\infty}(\Omega)}
			&\leq \sum_{i=1}^I \bigl\| |\nabla S_{m_i, m_{i-1}}(u^{\varepsilon})|^{p(x)-1} \bigr\|_{L^{N',\infty}(\Omega)}^{p}
			\leq \sum_{i=1}^I x_i \\
			&\leq b \sum_{i=1}^I (a+1)^{i-1}
			= b \frac{(a+1)^I - 1}{a}
			\leq \frac{b}{a} \bigl( (a+1)^I - 1 \bigr),
		\end{aligned}
	\end{equation}
	which is the desired result.
	From \eqref{tpkm}, we have
	\begin{equation}\label{uc1}
		\left\| |u^{\varepsilon}|^{p(x)-1} \right\|_{L^{N',\infty}(\Omega)} \leq C_1,
	\end{equation}
	\begin{equation}
		\left\| |\nabla u^{\varepsilon}|^{p(x)-1} \right\|_{L^{N',\infty}(\Omega)} \leq C_1,
	\end{equation}
	\begin{equation}
		\int_{\Omega} |\nabla T_k(u^{\varepsilon})|^{p(x)} \,\mathrm{d}x \leq C_1,
	\end{equation}
	where \( C_1 \) is a positive constant depending only on \( N, p^-, p^+, \alpha, \Omega, \|f^{\varepsilon}\|_{L^1(\Omega)}, \|b_1\|_{L^1(\Omega)}, \|b_0\|_{L^{N,1}(\Omega)} \).
	From \eqref{uc1}, we deduce that
	\begin{equation}
		T_k(u^{\varepsilon}) \text{ is bounded in } W_0^{1,p(x)}(\Omega),
	\end{equation}
	\begin{equation}
		T_k(u^{\varepsilon}) \rightharpoonup \delta_k \text{ weakly in } W_0^{1,p(x)}(\Omega),
	\end{equation}
	and
	\begin{equation}
		\nabla T_k(u^{\varepsilon}) \text{ is bounded in } L^{p(x)}(\Omega),
	\end{equation}
	\begin{equation}
		\nabla T_k(u^{\varepsilon}) \rightharpoonup \phi_k \text{ weakly in } L^{p(x)}(\Omega).
	\end{equation}

\section{Existence of Renormalized Solutions}
\begin{theorem}
	As \( \varepsilon \to 0 \), we have
	\begin{equation}
		A(x, T_k(u^\varepsilon), \nabla T_k(u^\varepsilon)) \cdot \nabla T_k(u^\varepsilon)
		\rightharpoonup A(x, T_k(u), \nabla T_k(u)) \cdot \nabla T_k(u)
		\quad \text{weakly in } L^1(\Omega).
	\end{equation}
\end{theorem}
We first introduce some auxiliary functions related to the test functions. 
\begin{remark}
	For any integer \( n \geq 1 \), define the bounded positive function
	\begin{equation}
		h_n(s) = 1 - \frac{|T_{2n}(s) - T_n(s)|}{n}.
	\end{equation}
	Taking \( (T_k(u^\varepsilon) - T_k(u)) h_n(u^\varepsilon) \) as a test function in equation \eqref{s1}, we obtain
	\begin{equation}
		\begin{aligned}
			&\int_{\Omega} A_\varepsilon(x, u^\varepsilon, \nabla u^\varepsilon) \, h_n(u^\varepsilon) \cdot \nabla (T_k(u^\varepsilon) - T_k(u)) \,\mathrm{d}x \\
			=& -\int_{\Omega} (T_k(u^\varepsilon) - T_k(u)) \, h_n'(u^\varepsilon) \, A_\varepsilon(x, u^\varepsilon, \nabla u^\varepsilon) \cdot \nabla u^\varepsilon \,\mathrm{d}x \\
			&- \int_{\Omega} H_\varepsilon(x, u^\varepsilon, \nabla u^\varepsilon) \, h_n(u^\varepsilon) \, (T_k(u^\varepsilon) - T_k(u)) \,\mathrm{d}x \\
			&+ \int_{\Omega} f^\varepsilon \, h_n(u^\varepsilon) \, (T_k(u^\varepsilon) - T_k(u)) \,\mathrm{d}x \\
			=& -I_1 - I_2 + I_3.
		\end{aligned}
	\end{equation}
	Since \( h_n \) has compact support, 
	$$
	T_k(u^\varepsilon) - T_k(u) \to 0 \quad a.e. \text{ in } \Omega , 
	$$
	and 
	$$
	T_k(u^\varepsilon) - T_k(u) \stackrel{*}{\rightharpoonup} 0 \text{ weakly$-*$ in } L^{\infty}(\Omega) .
	$$
	We have
	\begin{equation}
		\lim_{\varepsilon \to 0} \int_{\Omega} f^\varepsilon \, (T_k(u^\varepsilon) - T_k(u)) \, h_n(u^\varepsilon) \,\mathrm{d}x = 0.
	\end{equation}
	By a similar argument,
	\begin{equation}
		\lim_{\varepsilon \to 0} \int_{\Omega} H_\varepsilon(x, u^\varepsilon, \nabla u^\varepsilon) \, (T_k(u^\varepsilon) - T_k(u)) \, h_n(u^\varepsilon) \,\mathrm{d}x = 0.
	\end{equation}
	Moreover, since \( h_n'(r) = -\chi_{\{ n < |r| < 2n \}} \dfrac{\text{sign}(r)}{n} \) a.e. in \( \mathbb{R} \), and by the growth condition on \( A \), we have
	\begin{equation}
		|A_\varepsilon(x, u^\varepsilon, \nabla u^\varepsilon)| \leq \bigl[ |L(x)| + |\nabla u^\varepsilon|^{p(x)-1} \bigr]^{p(x)-1},
	\end{equation}
	where \( L(x) \in L^{p(x)-1}(\Omega) \).
\end{remark}
	We now estimate \( I_1 \). We have
	\begin{equation}
		\begin{aligned}
			|I_1|
			&= \Bigl| \int_{\Omega} (T_k(u^\varepsilon) - T_k(u)) \, h_n'(u^\varepsilon) \, A_\varepsilon(x, u^\varepsilon, \nabla u^\varepsilon) \cdot \nabla u^\varepsilon \,\mathrm{d}x \Bigr| \\
			&\leq \frac{2kC}{n} \Bigl[ \int_{\{ n < |u^\varepsilon| < 2n \}} |\nabla u^\varepsilon|^{p(x)} \,\mathrm{d}x + \int_{\Omega} |L(x)|^{p'(x)} \,\mathrm{d}x \Bigr].
		\end{aligned}
	\end{equation}
	It is easy to establish that
	\begin{equation}
		\lim_{n \to \infty} \limsup_{\varepsilon \to 0} \left| \int_{\Omega} (T_k(u^\varepsilon) - T_k(u)) \, h_n'(u^\varepsilon) \, A_\varepsilon(x, u^\varepsilon, \nabla u^\varepsilon) \cdot \nabla u^\varepsilon \,\mathrm{d}x \right| = 0.
	\end{equation}
	Based on the above estimates, we deduce that
	\begin{equation}\label{atkleq0}
		\lim_{n \to \infty} \limsup_{\varepsilon \to 0} \int_{\Omega} A_\varepsilon(x, u^\varepsilon, \nabla u^\varepsilon) \, h_n(u^\varepsilon) \cdot \nabla (T_k(u^\varepsilon) - T_k(u)) \,\mathrm{d}x \leq 0.
	\end{equation}
	By the definition of \( h_n \), for \( k \leq \dfrac{1}{\varepsilon} \) and \( k \leq n \), we have
	\begin{equation}\label{ahtat}
		A_\varepsilon(x, u^\varepsilon, \nabla u^\varepsilon) \, h_n(u^\varepsilon) \cdot \nabla T_k(u^\varepsilon)
		= A_\varepsilon(x, u^\varepsilon, \nabla u^\varepsilon) \cdot \nabla T_k(u^\varepsilon).
	\end{equation}
	Combining \eqref{atkleq0} and \eqref{ahtat}, we obtain
	\begin{equation}\label{ataht}
		\limsup_{\varepsilon \to 0} \int_{\Omega} A_\varepsilon(x, u^\varepsilon, \nabla u^\varepsilon) \cdot \nabla T_k(u^\varepsilon) \,\mathrm{d}x
		\leq \lim_{n \to \infty} \limsup_{\varepsilon \to 0} \int_{\Omega} A_\varepsilon(x, u^\varepsilon, \nabla u^\varepsilon) \, h_n(u^\varepsilon) \cdot \nabla T_k(u^\varepsilon) \,\mathrm{d}x.
	\end{equation}
	On the other hand, for \( \varepsilon \leq \dfrac{1}{2n} \), we have
	\begin{equation}
		A_\varepsilon(x, u^\varepsilon, \nabla u^\varepsilon) \, h_n(u^\varepsilon)
		= A(x, T_{2n}(u^\varepsilon), \nabla T_{2n}(u^\varepsilon)) \, h_n(u^\varepsilon),
		\quad \text{a.e. in } \Omega.
	\end{equation}
	By Remark \ref{aclremark}, for fixed \( n \geq 0 \), as \( \varepsilon \to 0 \) we have
	\begin{equation}
		A_\varepsilon(x, u^\varepsilon, \nabla u^\varepsilon) \, h_n(u^\varepsilon) \rightharpoonup \varphi_{2n} \, h_n(u), \quad \text{weakly in } L^{p'(x)}(\Omega).
	\end{equation}
	Using the above weak convergence, for \( k \leq n \) we obtain
	\begin{equation}\label{ahtphi}
		\lim_{\varepsilon \to 0} \int_{\Omega} A_\varepsilon(x, u^\varepsilon, \nabla u^\varepsilon) \, h_n(u^\varepsilon) \cdot \nabla T_k(u) \,\mathrm{d}x
		= \int_{\Omega} \varphi_{2n} \, h_n(u) \cdot \nabla T_k(u) \,\mathrm{d}x.
	\end{equation}
	Indeed, for \( k \leq n \), we have
	\begin{equation}
		A(x, T_{2n}(u^\varepsilon), \nabla T_{2n}(u^\varepsilon)) \, \chi_{\{ |u^\varepsilon| < k \}}
		= A(x, T_k(u^\varepsilon), \nabla T_k(u^\varepsilon)) \, \chi_{\{ |u^\varepsilon| < k \}},
		\quad \text{a.e. in } \Omega.
	\end{equation}
	From \eqref{tkdelta} and \eqref{82}, we deduce that
	\begin{equation}\label{phit}
		\varphi_{2n} \, \chi_{\{ |u| < k \}} = \varphi_k \, \chi_{\{ |u| < k \}}, \quad \text{a.e. in } \Omega \setminus \{ |u| = k \}.
	\end{equation}
	Since \( \nabla T_k(u) = 0 \) almost everywhere on \( \{ |u| = k \} \), we have
	\begin{equation}
		\varphi_{2n} \cdot \nabla T_k(u) = \varphi_k \cdot \nabla T_k(u), \quad \text{a.e. in } \Omega.
	\end{equation}
	From \eqref{ataht}, \eqref{ahtphi} and \eqref{phit}, we obtain
	\begin{equation}
		\lim_{\varepsilon \to 0} \int_{\Omega} A(x, u^\varepsilon, \nabla u^\varepsilon) \cdot \nabla T_k(u^\varepsilon) \,\mathrm{d}x
		\leq \int_{\Omega} \varphi_k \cdot \nabla T_k(u) \,\mathrm{d}x.
	\end{equation}
	Fix \( k \geq 0 \). From \eqref{as3} we have
	\begin{equation}\label{aatt1}
		\lim_{\varepsilon \to 0} \int_{\Omega} \bigl[ A(x, T_k(u^\varepsilon), \nabla T_k(u^\varepsilon)) - A(x, T_k(u^\varepsilon), \nabla T_k(u)) \bigr]
		\cdot \bigl[ \nabla T_k(u^\varepsilon) - \nabla T_k(u) \bigr] \,\mathrm{d}x \geq 0.
	\end{equation}
	Now, from \eqref{as4} and \eqref{tkdelta}, we have
	\begin{equation}
		A_\varepsilon(x, T_k(u^\varepsilon), \nabla T_k(u^\varepsilon)) \to A(x, T_k(u), \nabla T_k(u)), \quad \text{a.e. in } \Omega,
	\end{equation}
	and
	\begin{equation}
		|A(x, T_k(u), \nabla T_k(u))| \leq \bigl[ |L(x)| + |\nabla T_k(u)|^{p(x)-1} \bigr].
	\end{equation}
	For every \( \varepsilon < \dfrac{1}{k} \), we obtain
	\begin{equation}\label{aa}
		A_\varepsilon(x, T_k(u^\varepsilon), \nabla T_k(u^\varepsilon)) \to A(x, T_k(u), \nabla T_k(u)), \quad \text{in } L^{p'(x)}(\Omega).
	\end{equation}
	By \eqref{80}, \eqref{aatt1} and \eqref{aa}, together with the generalized Hölder inequality, we have
	\begin{equation}\label{aatt2}
		\lim_{\varepsilon \to 0} \int_{\Omega} \bigl[ A(x, T_k(u^\varepsilon), \nabla T_k(u^\varepsilon)) - A(x, T_k(u^\varepsilon), \nabla T_k(u)) \bigr]
		\cdot \bigl[ \nabla T_k(u^\varepsilon) - \nabla T_k(u) \bigr] \,\mathrm{d}x = 0.
	\end{equation}
	For any \( k > 0 \), \( 0 < \varepsilon < \dfrac{1}{k} \), and \( \psi \in \mathbb{R}^N \), we have
	\begin{equation}
		A_\varepsilon(x, T_k(u^\varepsilon), \psi) = A(x, T_k(u^\varepsilon), \psi) = A_{\frac{1}{k}}(x, T_k(u^\varepsilon), \psi).
	\end{equation}
	Combining \eqref{82}, \eqref{aa} and \eqref{aatt2}, we obtain
	\begin{equation}\label{atphit}
		\lim_{\varepsilon \to 0} \int_{\Omega} A_1(x, T_k(u^\varepsilon), \nabla T_k(u^\varepsilon)) \cdot \nabla T_k(u^\varepsilon) \,\mathrm{d}x
		= \int_{\Omega} \varphi_k \cdot \nabla T_k(u) \,\mathrm{d}x.
	\end{equation}
	For any \( k > 0 \), the function \( A_1(x, s, \psi) \) is continuous and bounded in \( s \). From \eqref{77}, \eqref{82} and \eqref{atphit}, we conclude that
	\begin{equation}\label{102}
		\varphi_k = A(x, T_k(u), \nabla T_k(u)), \quad \text{a.e. in } \Omega.
	\end{equation}
	Therefore, for \( k > 0 \), as \( \varepsilon \to 0 \), we have in \( L^1(\Omega) \)
	\begin{equation}\label{aattag}
		\bigl[ A(x, T_k(u^\varepsilon), \nabla T_k(u^\varepsilon)) - A(x, T_k(u^\varepsilon), \nabla T_k(u)) \bigr]
		\cdot \bigl[ \nabla T_k(u^\varepsilon) - \nabla T_k(u) \bigr] \to 0,
		\quad \text{ as } \varepsilon \to 0.
	\end{equation}
	From \eqref{77}, \eqref{80}, and \eqref{aa}-\eqref{atphit}, we obtain, in \( L^1(\Omega) \),
	\begin{equation}
		A(x, T_k(u^\varepsilon), \nabla T_k(u^\varepsilon)) \cdot \nabla T_k(u)
		\rightharpoonup A(x, T_k(u), \nabla T_k(u)) \cdot \nabla T_k(u),
	\end{equation}
	\begin{equation}
		A(x, T_k(u^\varepsilon), \nabla T_k(u)) \cdot \nabla T_k(u^\varepsilon)
		\rightharpoonup A(x, T_k(u), \nabla T_k(u)) \cdot \nabla T_k(u),
	\end{equation}
	\begin{equation}
		A(x, T_k(u^\varepsilon), \nabla T_k(u)) \cdot \nabla T_k(u)
		\rightarrow A(x, T_k(u), \nabla T_k(u)) \cdot \nabla T_k(u)
	\end{equation}
	as \(\varepsilon \to 0\).
	Combining the above three convergence results with \eqref{aattag}, we obtain, in \( L^1(\Omega) \),
	\begin{equation}\label{104}
		A(x, T_k(u^\varepsilon), \nabla T_k(u^\varepsilon)) \cdot \nabla T_k(u^\varepsilon)
		\to A(x, T_k(u), \nabla T_k(u)) \cdot \nabla T_k(u).
	\end{equation}
	From \eqref{77}, using Lemma~5 of \cite{Yazough2018}, we have
	\begin{equation}\label{cite}
		T_k(u^\varepsilon) \to T_k(u), \quad \text{in } W_0^{1,p(x)}(\Omega).
	\end{equation}
	For any \( \delta > 0 \), the following holds:
	\begin{equation}
		\begin{aligned}[b]
			\operatorname{meas}\{ x \in \Omega : |\nabla u^\varepsilon - \nabla u| > \delta \}
			&\leq \operatorname{meas}\{ x \in \Omega : |u^\varepsilon| > k \} \\
			&\quad + \operatorname{meas}\{ x \in \Omega : |u| > k \} \\
			&\quad + \operatorname{meas}\{ x \in \Omega : |\nabla T_k(u^\varepsilon) - \nabla T_k(u)| > \delta \}.
		\end{aligned}
	\end{equation}
	Letting \( k \to \infty \), using \( \displaystyle \lim_{k \to \infty} \operatorname{meas}\{ x \in \Omega : |u^\varepsilon| > k \} = 0 \) and \eqref{cite}, we obtain
	\begin{equation}
		\nabla u^\varepsilon \to \nabla u, \quad \text{a.e. in } \Omega.
	\end{equation}
\begin{theorem}
	There exists a renormalized solution \( u \) of the original problem \eqref{originalequation}.
\end{theorem}
\begin{proof}
	For any function \( S \in W^{2,\infty}(\mathbb{R}) \) such that \( S \) is pointwise \( C^1 \) and \( S' \) has compact support, we take \( S'(u^\varepsilon) \phi \) as a test function in \eqref{s1}, where \( \phi \in C_0^\infty(\Omega) \). This yields
	\begin{equation}\label{106}
		\begin{aligned}[b]
			&-\nabla\cdot \bigl( S'(u^\varepsilon) A_\varepsilon(x, u^\varepsilon, \nabla u^\varepsilon) \bigr) + S''(u^\varepsilon) A_\varepsilon(x, u^\varepsilon, \nabla u^\varepsilon) \cdot \nabla u^\varepsilon \\
			&\quad + H_\varepsilon(x, u^\varepsilon, \nabla u^\varepsilon) S'(u^\varepsilon) = f^\varepsilon S'(u^\varepsilon).
		\end{aligned}
	\end{equation}
	From \eqref{tkdelta}, we have
	\begin{equation}\label{ss}
		S'(u^\varepsilon) \stackrel{*}{\rightharpoonup} S'(u), \quad \text{a.e. in } \Omega, \ \text{weakly-\(*\) in } L^\infty(\Omega),
	\end{equation}
	\begin{equation}\label{ss1}
		S''(u^\varepsilon) \stackrel{*}{\rightharpoonup} S''(u), \quad \text{a.e. in } \Omega, \ \text{weakly-\(*\) in } L^\infty(\Omega).
	\end{equation}
	Since \( \operatorname{supp} S' \subset [-k,k] \) and \( S', S'' \) are bounded, for \( 0 < \varepsilon < \dfrac{1}{k} \) we have
	\begin{equation}
		S'(u^\varepsilon) A(x, T_{\frac{1}{\varepsilon}}(u^\varepsilon), \nabla u^\varepsilon)
		= S'(u^\varepsilon) A(x, T_k(u^\varepsilon), \nabla T_k(u^\varepsilon)),
	\end{equation}
	\begin{equation}\label{110}
		S''(u^\varepsilon) A(x, T_{\frac{1}{\varepsilon}}(u^\varepsilon), \nabla u^\varepsilon) \cdot \nabla u^\varepsilon
		= S''(u^\varepsilon) A(x, T_k(u^\varepsilon), \nabla T_k(u^\varepsilon)) \cdot \nabla T_k(u^\varepsilon).
	\end{equation}
	From \eqref{82}, \eqref{102}, \eqref{ss} and \eqref{110}, we obtain
	\begin{equation}
		S'(u^\varepsilon) A(x, T_{\frac{1}{\varepsilon}}(u^\varepsilon), \nabla u^\varepsilon)
		\rightharpoonup S'(u) A(x, T_k(u), \nabla T_k(u)), \quad \text{weakly in } L^{p'(x)}(\Omega).
	\end{equation}
	Moreover, by \eqref{104}, \eqref{ss1} and \eqref{110}, in \( L^1(\Omega) \) we have
	\begin{equation}
		S''(u^\varepsilon) A(x, T_k(u^\varepsilon), \nabla T_k(u^\varepsilon)) \cdot \nabla T_k(u^\varepsilon)
		\to S''(u) A(x, T_k(u), \nabla T_k(u)) \cdot \nabla T_k(u).
	\end{equation}
	Since 
 	\begin{equation}
 		H_\varepsilon(x, u^\varepsilon, \nabla u^\varepsilon) \rightharpoonup \Lambda \text{ weakly in } L^1(\Omega) , 
 	\end{equation}
	we have
	\begin{equation}
		\Lambda = H(x, u, \nabla u), \quad \text{a.e. in } \Omega.
	\end{equation}
	Consequently,
	\begin{equation}
		S'(u^\varepsilon) H_\varepsilon(x, u^\varepsilon, \nabla u^\varepsilon)
		\rightharpoonup S'(u) H(x, u, \nabla u), \quad \text{weakly in } L^1(\Omega).
	\end{equation}
	Since \( f^\varepsilon \to f \) strongly in \( L^1(\Omega) \), we have
	\begin{equation}
		f^\varepsilon S'(u^\varepsilon) \to f S'(u), \quad \text{strongly in } L^1(\Omega).
	\end{equation}
	Passing to the limit as \( \varepsilon \to 0 \) in \eqref{106} and using the above convergence results, we obtain \eqref{as8}. Hence \( u \) is a renormalized solution of problem \eqref{originalequation}.
\end{proof}

\section*{Acknowledgments}
	
\newpage
\bibliographystyle{plain}
\bibliography{refer.bib}
\end{document}